\documentclass[12pt]{amsart}
\usepackage{graphicx}
\usepackage{amsmath, amssymb, epsfig}
\usepackage{amsmath, amssymb}
\usepackage{mathrsfs}
\usepackage{color}
\usepackage[numbers]{natbib}
\usepackage{url}
\usepackage{bm} 
\usepackage{bbm}
\usepackage{alg}

\makeatletter
\def\algspacing{\alg@unmargin}
\makeatother

\usepackage[margin=1in]{geometry}
\newlength{\algorithmwidth}
\theoremstyle{plain}
\newtheorem{theorem}{Theorem}[section]

\newtheorem{corollary}[theorem]{Corollary}
\newtheorem{lemma}[theorem]{Lemma}

\theoremstyle{definition}

\theoremstyle{remark}

\numberwithin{theorem}{section}
\numberwithin{equation}{section}
\newcommand{\<}{\left\langle}
\renewcommand{\>}{\right\rangle}

\DeclareMathOperator*{\argmin}{arg min} 
 
\DeclareMathOperator*{\variance}{Var}

\def \P {\mathbb{P}}
\def \E {\mathbb{E}}
\def \A {\mathcal{A}}

\def \F {\mathcal{G}}

\def \R {\mathcal{R}}

\newcommand{\vct}[1]{\bm{#1}}
\newcommand{\mtx}[1]{\bm{#1}}
\newcommand{\ip}[2]{\left\langle {#1}, \, {#2} \right\rangle}
\newcommand{\norm}[1]{\lVert {#1} \rVert}

\newcommand{\enorm}[1]{\norm{#1}_2}
\newcommand{\enormsq}[1]{\enorm{#1}^2}

\usepackage{times}

\def \xcur {\vct{x}_k}
\def \xnext {\vct{x}_{k+1}}
\def \xmin {\vct{x}_{\star}}
\def \ai {\vct{a}_i}
\def \bi {b_i}
\def \F {F} 
\newcommand{\gradest}{\vct{g}}

\def \x {\vct{x}}
\def \y {\vct{y}}
\def \A {\mtx{A}}
\def \D {\mtx{D}}

\def \b {\vct{b}}

\def \z {\vct{z}}
\def \e {\vct{e}}
\def \xhat {\hat{\x}}
\def \xo {\vct{x}_0}
\def \Fradg {\nabla\vct{F}}  
\newcommand{\DD}{\mathcal{D}}  
\newcommand{\fwi}{f^{(w)}_i}
\newcommand{\Lwi}{L^{(w)}_i}
\newcommand{\Fwi}{G^{(w)}_i}
\newcommand{\supLw}{\sup L_{(w)}}
\newcommand{\wL}{w}  %
\newcommand{\wLA}{w^{\lambda}}
\newcommand{\supLwL}{\sup L_{(\wL)}}
\newcommand{\supLwLA}{\sup L_{(\wLA)}}
\newcommand{\sigmaLwL}{\sigma_{(w)}}
  \newcommand{\sigmaLwLA}{\sigma_{(w^{\lambda})}}
\newcommand{\sqLw}{\overline{L_{(w)}^2}}
\newcommand{\sqGw}{\overline{G_{(w)}^2}}

\newcommand{\removed}[1]{{}}

\begin{document}

\title[SGD and Kaczmarz]{Stochastic Gradient Descent, Weighted Sampling, and
  the Randomized Kaczmarz algorithm}

\author[Needell, Srebro, Ward]{Deanna Needell \and\ Nathan Srebro \and\ Rachel ~Ward}

\thanks{DN: Department of Mathematical Sciences, Claremont McKenna College (dneedell@cmc.edu). NS: Toyota Technological Institute at Chicago (nati@ttic.edu). RW: Department of Mathematics, University of Texas at Austin (rward@math.utexas.edu).}
\date{\today}

\begin{abstract}
  We obtain an improved finite-sample guarantee on the linear
  convergence of stochastic gradient descent for smooth and strongly
  convex objectives, improving from a quadratic dependence on the
  conditioning $(L/\mu)^2$ (where $L$ is a bound on the smoothness and
  $\mu$ on the strong convexity) to a linear dependence on $L/\mu$.
  Furthermore, we show how reweighting the sampling distribution
  (i.e.~importance sampling) is necessary in order to further improve
  convergence, and obtain a linear dependence in the average
  smoothness, dominating previous results.  We also discuss importance
  sampling for SGD more broadly and show how it can improve
  convergence also in other scenarios.

  Our results are based on a connection we make between SGD and the
  \emph{randomized Kaczmarz algorithm}, which allows us to transfer
  ideas between the separate bodies of literature studying each of the
  two methods.    In particular, we recast the 
 randomized Kaczmarz algorithm as an instance of SGD, and apply our results to prove its exponential convergence, but to the solution of a weighted least squares problem rather than the original least squares problem.  We then present a modified Kaczmarz algorithm with partially biased sampling which does converge to the original least squares solution with the same exponential convergence rate. 
  

\end{abstract}

\maketitle

\section{Introduction} 
This paper connects two algorithms which until now have remained
remarkably disjoint in the literature: the randomized Kaczmarz
algorithm for solving linear systems and the stochastic gradient
descent (SGD) method for optimizing a convex objective using unbiased
gradient estimates.  The connection enables us to make contributions
by borrowing from each body of literature to the other.  In
particular, it helps us highlight the role of weighted sampling for
SGD and obtain a tighter guarantee on the linear convergence regime of
SGD.

\removed{
We extend the idea of importance sampling from the Kaczmarz
literature, and introduce an family of SGD algorithms with
preconditioning and nonuniform row selection rule.  We show that by
sampling from a hybrid uniform / biased distribution over the
estimates, the proposed family of algorithms enjoys the advantages of
both importance sampling (improved rate of convergence) and unbiased
sampling (higher noise tolerance), and overall compares favorably to
recent bounds for SGD with uniform row selection in the case of sums
of smooth and strongly convex functions.
}

Recall that stochastic gradient descent is a method for minimizing a
convex objective $\F(\x)$ based on access to unbiased stochastic
gradient estimates, i.e.~to an estimate $\gradest$ for the gradient at
a given point $\x$, such that $\E[\gradest]=\nabla \F(\x)$.  Viewing
$\F(\x)$ as an expectation $\F(\x)=\E_i[f_i(\x)]$, the unbiased
gradient estimate can be obtained by drawing $i$ and using its
gradient: $\gradest=\nabla f_i(\x)$.  SGD originated as ``Stochastic
Approximation'' in the pioneering work of \citet{robmon}, and has
recently received renewed attention for confronting very large scale
problems, especially in the context of machine learning
\citep{bousquet2007tradeoffs,shalev2008svm,njls09,bottou2010large}.
Classical analysis of SGD shows a polynomial rate on the
sub-optimality of the objective value, $\F(\xcur)-\F(\xmin)$, namely
$1/\sqrt{k}$ for non-smooth objectives, and $1/k$ for smooth, or
non-smooth but strongly convex objectives.  Such convergence can be
ensured even if the iterates $\xcur$ do not necessarily converge to a
unique optimum $\xmin$, as might be the case if $\F(\x)$ is not
strongly convex.  Here we focus on the strongly convex case, where the
optimum is unique, and on convergence of the iterates $\xcur$ to the
optimum $\xmin$.

\citet{bach2011} recently provided a non-asymptotic
bound on the convergence of the iterates in strongly convex SGD,
improving on previous results of this kind \citep[Section
2.2]{murata}\citep[Section 3.2]{bottou}\citep{shamir2012stochastic}\citep{njls09}.
In particular, they showed that if each $f_i(\x)$ is
smooth and if $\xmin$ is a minimizer of (almost) all $f_i(\x)$,
i.e.~$\P_i(\nabla f_i(\xmin)=0)=1$, then $\E\norm{\xcur-\xmin}$ goes
to zero exponentially, rather than polynomially, in $k$.  That is, reaching
a desired accuracy of $\E \norm{\xcur-\xmin}^2 \leq \varepsilon$
requires a number of steps that scales only logarithmically in
$1/\varepsilon$.  Bach and Moulines's bound on the required number of
iterations further depends on the average \emph{squared} conditioning
number $\E[(L_i/\mu)^2]$, where $L_i$ is the Lipschitz
constant of $\nabla f_i(\x)$ (i.e.~$f_i(\x)$ are ``$L$-smooth''), and
$\F(\x)$ is $\mu$-strongly convex.  If $\xmin$ is not an exact
minimizer of each $f_i(\x)$, the bound degrades gracefully as a
function of $\sigma^2=\E \norm{\nabla f_i(\xmin)}^2$, and includes an
unavoidable term that behaves as $\sigma^2/k$.

In a seemingly independent line of research, the \emph{Kaczmarz
  method} was proposed as an iterative method for solving (usually
overdetermined) systems of linear
equations~\citep{Kac37:Angenaeherte-Aufloesung}.  The simplicity of the
method makes it useful in a wide array of applications ranging from
computer tomography to digital signal
processing~\citep{herman2009fundamentals,Nat01:Mathematics-Computerized,hounsfield1973computerized}.
Recently, \citet{SV09:Randomized-Kaczmarz} proposed a variant of the
Kaczmarz method using a random selection method which select rows with
probability proportional to their squared norm, and showed that using
this selection strategy, a desired accuracy of $\varepsilon$ can be
reached in the noiseless setting in a number of steps that scales like
$ \log(1/\varepsilon)$ and {\em linearly} in the condition number.

\subsection{Importance sampling in stochastic optimization}

From a birds-eye perspective, this paper aims to extend the notion of importance sampling from stochastic sampling methods for numerical linear algebra applications, to more general stochastic convex optimization problems. 
Strohmer and Vershynin's incorporation of importance sampling into the Kaczmarz setup ~\cite{SV09:Randomized-Kaczmarz} is just one such example, and most closely related to the SGD set-up.   But importance sampling has also been considered in stochastic coordinate-descent methods~\cite{nesterov2012efficiency,richtarik2012iteration}. There also, the weights are proportional to some power of the Lipschitz constants (of the gradient coordinates).

Importance sampling has also played  a key role in designing sampling-based low-rank matrix approximation algorithms   -- both row/column based sampling and entry-wise sampling -- where it goes by the name of \emph{leverage score} sampling.  The resulting sampling methods are again proportional to the squared Euclidean norms of rows and columns of the underlying matrix.  See \cite{boutsidis2009columnselection, mahoney2011, spielman2011graphsparsify, chen2014coherent}, and references therein for applications to the column subset selection problem and matrix completion. 
See \cite{gittensmahoney, mamahoney14, wangzhang13} for applications of importance sampling to the Nystr{\"o}m Method. 

Importance sampling has also been introduced to the compressive sensing framework, where it translates to sampling rows of an orthonormal matrix proportionally to their squared inner products with the rows of a second orthonormal matrix in which the underlying signal is assumed sparse. See \cite{rauhut2012sparse, krahmer2012beyond} for more details.

\subsection{Contributions}

Inspired by the analysis of \citeauthor{SV09:Randomized-Kaczmarz} and \citeauthor{bach2011}, we prove convergence results for stochastic
gradient descent as well as for SGD variants where gradient
estimates are chosen based on a {\em weighted sampling distribution},
highlighting the role of importance sampling in SGD.

We first show (Corollary \ref{thm:SGDsteps} in Section \ref{sec:SGD})
that without perturbing the sampling distribution, we can obtain a
linear dependence on the {\em uniform conditioning} $(\sup L_i /
\mu)$, but it is {\em not possible} to obtain a linear dependence on
the {\em average conditioning} $\E[L_i/\mu]$.  This is a {\bfseries
  quadratic improvement over the previous results \citep{bach2011} in regimes where the
  components have similar Lipschitz constants}.

We then turn to importance sampling, using a weighted sampling
distribution.  We show that weighting components proportionally to
their Lipschitz constants $L_i$, as is essentially done by Strohmer
and Vershynin, can reduce the dependence on the conditioning to a
linear dependence on the average conditioning $\E[L_i/\mu]$.  However,
this comes at an increased dependence on the residual $\sigma^2$.
But, we show that by only {\em partially biasing} the sampling towards
$L_i$, we can enjoy the best of both worlds, obtaining a linear
dependence on the average conditioning $\E[L_i/\mu]$, without
amplifying the dependence on the residual. Thus, {\bfseries using
  importance sampling, we obtain a guarantee dominating, and improving
  over the previous best-known results~\citep{bach2011}} (Corollary \ref{thm:mixsteps} in
Section \ref{sec:SGD}).

In Section \ref{sec:other}, we consider the benefits of reweighted SGD
also in other scenarios and regimes.  We show how also for smooth but
not-strongly-convex objectives, {\bfseries importance sampling can
  improve a dependence on a uniform bound over smoothness, $(\sup
  L_i)$, to a dependence on the average smoothness $\E[L_i]$}---such
an improvement is not possible without importance sampling.  For
non-smooth objectives, we show that importance sampling can eliminate
a dependence on the variance in the Lipschitz constants of the
components.  In parallel work we recently became aware of,
\citet{zhao2014stochastic} also consider importance sampling for
non-smooth objectives, including composite objectives, suggesting the
same reweighting as we obtain here.

Finally, in Section \ref{sec:kz}, we turn to the Kaczmarz algorithm, explain
how it is an instantiation of SGD, and how using partially biased
sampling improves known guarantees in this context as well.  We show that the randomized Kaczmarz method with uniform i.i.d. row selection can be recast as an instance of preconditioned Stochastic Gradient Descent acting on a re-weighted least squares problem and through this connection, provide exponential convergence rates for this algorithm.  We also consider the Kaczmarz algorithm corresponding to SGD with hybrid row selection strategy which shares the exponential convergence rate of \citet{SV09:Randomized-Kaczmarz} while also sharing a small error residual term of the SGD algorithm.  This presents a clear tradeoff between convergence rate and the convergence residual, not present in other results for the method.

\section{SGD for Strongly Convex Smooth Optimization}\label{sec:SGD}

We consider the problem of minimizing a smooth convex function,
\begin{equation}\label{eq:min2}
\xmin = \argmin_{\x} \F(\x)
\end{equation}
where $\F(\x)$ is of the form $\F(\x) = \E_{i\sim\DD} f_i(\x)$ for
smooth functionals $f_i:\mathcal{H}\rightarrow\R$ over
$\mathcal{H}=\R^d$ endowed with the standard Euclidean norm
$\norm{\cdot}_2$, or over a Hilbert space $\mathcal{H}$ with the norm
$\norm{\cdot}_2$.  Here $i$ is drawn from some {\em source
  distribution} $\DD$ over an arbitrary probability space.  Throughout
this manuscript, unless explicitly specified otherwise, expectations
will be with respect to indices drawn from the source distribution
$\DD$.  That is, we write $\E f_i(\x)=\E_{i\sim\DD} f_i(\x)$.  We also
denote by $\sigma^2$ the ``residual'' quantity at the minimum,
$$
\sigma^2= \E \enormsq{\nabla f_i(\xmin)}.
$$

We will instate the following assumptions on the function $\F$:
\begin{enumerate}
\item Each $f_i$ is continuously differentiable and the gradient
  function $\nabla f_i$ has Lipschitz constant $L_i$; that is,
  $\enorm{\nabla f_i(\x) - \nabla f_i(\y)} \leq L_i\enorm{\x-\y}$ for all
  vectors $\x$ and $\y$.\vspace{2mm}
\item $\F$ has strong convexity parameter $\mu$; that is, $\<\x-\y, \Fradg(\x) - \Fradg(\y)\> \geq \mu\enormsq{\x-\y}$ for all vectors $\x$ and $\y$. 
\end{enumerate}
\removed{Note in particular that the strong convexity assumption ensures that the minimum of \eqref{eq:min2} is unique.}

We denote $\sup L$ the supremum of the support of $L_i$, i.e.~the
smallest $L$ such that $L_i \leq L$ a.s., and similarly denote $\inf
L$ the infimum.  We denote the average Lipschitz constant as $\overline{L}=\E L_i$.

A unbiased gradient estimate for $F(\x)$ can be obtained by
drawing $i\sim\DD$ and using $\nabla f_i(\x)$ as the estimate.  The SGD
updates with (fixed) step size $\gamma$ based on these gradient
estimates are then given by:
\begin{equation}
  \label{eq:SGD}
  \x_{k+1} \leftarrow \x_k - \gamma \nabla f_{i_k}(\x_k)
\end{equation}
where $\{i_k\}$ are drawn i.i.d.~from $\DD$.  We are interested in the
distance $\enormsq{\xcur-\xmin}$ of the iterates from the unique
minimum, and denote the initial distance by $\varepsilon_0 = \enormsq{\x_0-\xmin}$.

\citet[Theorem 1]{bach2011} considered this
setting\footnote{Bach and Moulines's results are somewhat more
  general.  Their Lipschitz requirement is a bit weaker and more
  complicated, but in terms of $L_i$ yields \eqref{eq:bmbound}.  They
  also study the use of polynomial decaying step-sizes, but these do
  not lead to improved runtime if the target accuracy is known ahead
  of time. } and established that
\begin{equation}\label{eq:bmbound}
k = 2\log(\varepsilon/\varepsilon_0) \Big(\frac{\E L_i^2}{\mu^2} + \frac{\sigma^2}{\mu^2 \varepsilon} \Big) 
\end{equation}
SGD iterations of the form \eqref{eq:SGD}, with an appropriate
step-size, are sufficient to ensure
$\E\enormsq{\xcur-\xmin}\leq\varepsilon$, where the expectations is
over the random sampling.  As long as $\sigma^2=0$, i.e.~the same minimizer
$\xmin$ minimizes all components $f_i(\x)$ (though of course it need
not be a unique minimizer of any of them), this yields linear
convergence to $\xmin$, with a graceful degradation as
$\sigma^2>0$.  However, in the linear convergence regime, the number of
required iterations scales with the expected \emph{squared} conditioning
$\E L_i^2/\mu^2$.  In this paper, we reduce this quadratic dependence to a linear
dependence.  We begin with a guarantee ensuring linear dependence,
though with a dependence on $\sup L/\mu$ rather than $E L_i/\mu$:

\begin{theorem}\label{thm:gen}
  Let each $f_i$ be convex where $\nabla f_i$ has Lipschitz constant
  $L_i$, with $L_i\leq \sup L$ a.s., and let $\F(\x)=\E f_i(\x)$ be
  $\mu$-strongly convex.  Set $\sigma^2 = \E
  \enormsq{\nabla f_i(\xmin)}$, where $\xmin = \argmin_{\x} \F(\x)$.
 Suppose that $\gamma < \frac{1}{ \sup L}$.  Then the SGD iterates given by \eqref{eq:SGD} satisfy:
\begin{equation}
\label{recursion}
\E\enormsq{\xcur - \xmin} \leq \left[ 1 - 2\gamma \mu(1 - \gamma \sup
  L) \Big) \right]^k\enormsq{\xo - \xmin} + \frac{\gamma
  \sigma^2}{\mu\big( 1 - \gamma \sup L \big)},
\end{equation}
where the expectation is with respect to the sampling of $\{i_k\}$. 
\end{theorem}

If we are given a desired tolerance, $ \enormsq{\x - \xmin} \leq
\varepsilon$, and we know the Lipschitz constants and parameters of
strong convexity, we may optimize the step-size $\gamma$, and obtain:

\begin{corollary}\label{thm:SGDsteps} For any desired $\varepsilon$,
  using a step-size of
$$ \gamma = \frac{\mu \varepsilon}{2 \varepsilon \mu \sup L + 2
  \sigma^2}$$
we have that after
{
\begin{equation}\label{eq:cor1}
k = 2\log(2\varepsilon_0/\varepsilon) \Big(\frac{\sup L}{\mu} + \frac{\sigma^2}{\mu^2 \varepsilon} \Big) 
\end{equation}
}
SGD iterations, $\E\enormsq{\xcur - \xmin} \leq \varepsilon$, where $\varepsilon_0 = \enormsq{\x_0 - \xmin}$ and where the expectation is with respect to the sampling of $\{i_k\}$.
\end{corollary}

\begin{proof}
Substituting $\gamma = \frac{\mu \varepsilon}{2 \varepsilon \mu \sup L + 2\sigma^2}$ into the second term of \eqref{recursion} and simplifying gives the bound
  $$
  \frac{\gamma \sigma^2}{\mu\big( 1 - \gamma \sup L \big)} \leq \varepsilon/2.
  $$
	{
 Now asking that
  $$
  \left[ 1 - 2\gamma \mu(1 - \gamma \sup
  L) \Big) \right]^k \varepsilon_0  \leq \varepsilon/2,
  $$
  substituting for $\gamma$, and rearranging to solve for $k$, shows that we need $k$ such that
	$$
	k\log\left(1 - \frac{\mu^2\varepsilon(\mu\varepsilon\sup L + 2\sigma^2)}{2(\mu\varepsilon\sup L + \sigma^2)^2} \right) \leq -\log\left(\frac{2\varepsilon_0}{\varepsilon}\right).
	$$
	Utilizing the fact that $-1/\log (1-x) \leq 1/x$ for $0 < x \leq 1$ and rearranging again yields the requirement that 
	$$
	k \geq \log\left(\frac{2\varepsilon_0}{\varepsilon}\right)\cdot\frac{2(\mu\varepsilon\sup L + \sigma^2)^2}{\mu^2\varepsilon(\mu\varepsilon\sup L + 2\sigma^2)}.
	$$
	Noting that this inequality holds when $k \geq 2\log\left(\frac{2\varepsilon_0}{\varepsilon}\right)\cdot\frac{\mu\varepsilon\sup L + \sigma^2}{\mu^2\varepsilon}$ yields
	the stated number of steps $k$ in \eqref{eq:cor1}.  Since the expression on the right hand side of \eqref{recursion} decreases with $k$, the corollary is proven.
	}
  \end{proof}

\paragraph{\bf Proof sketch.} The crux of the improvement over Bach and
Moulines is in a tighter recursive equation.  Bach and
Moulines rely on the recursion
$$
\enormsq{\xnext -\xmin} \leq  \left(1 - 2 \gamma \mu + 2 \gamma^2 L_i^2 \right)\enormsq{\xcur - \xmin} + 2\gamma^2\sigma^2,
$$
whereas we use the Co-Coercivity Lemma~\ref{cocoercivity}, with which we can obtain the recursion
$$
\enormsq{\xnext -\xmin} \leq  \left(1 - 2 \gamma \mu + 2\gamma^2 \mu L_i\right)\enormsq{\xcur - \xmin} + 2\gamma^2\sigma^2,
$$
where $L_i$ is the Lipschitz constant of the component used in the
current iterate.  The significant difference is that one of the
factors of $L_i$ (an upper bound on the second derivative), in the third
term inside the parenthesis, is replaced by $\mu$ (a lower bound on
the second derivative of $F$).  A complete proof can be found in the appendix.

\bigskip

\paragraph{\bf Comparison to results of~\citeauthor{bach2011}.} Our bound \eqref{eq:cor1}
replaces the dependence on the average {\em square}
conditioning $(\E L_i^2/\mu^2)$ with a linear dependence on the {\em
  uniform} conditioning $(\sup L/\mu)$.  When all Lipschitz constants
$L_i$ are of similar magnitude, this is a quadratic improvement in the
number of required iterations.  However, when different components
$f_i$ have widely different scaling, i.e.~$L_i$ are highly variable,
the supremum might be larger then the average square
conditioning.

\bigskip

\paragraph{\bf Tightness.} Considering the above, one might hope to
obtain a linear dependence on the average conditioning
$\overline{L}/\mu = \E L_i/\mu$.  However, as the following example
shows, this is not possible.  Consider a uniform source distribution
over $N+1$ quadratics, with the first quadratic $f_1$ being
$\frac{N}{2} (\x[1]-b)^2$ and all others being $\frac{1}{2}\x[2]^2$,
and $b = \pm 1$.  Any method must examine $f_1$ in order to recover
$\x$ to within error less then one, but by uniformly sampling indices
$i$, this takes $(N+1)$ iterations in expectation.  It is easy to
verify that in this case, $\sup L_i=L_1=N$,
$\overline{L}=2\frac{N}{N+1}<2$ $\E L^2_i =N,$ and
$\mu=\frac{N}{N+1}$.  For large $N$, a linear dependence on
$\overline{L}/\mu$ would mean that a constant number of iterations
suffice (as $\overline{L}/\mu=2$), but we just saw that {\em any}
method that sampled $i$ uniformly must consider at least $(N+1)$
samples in expectation to get non-trivial error.  Note that both $\sup
L_i/\mu=N+1$ and {$\E L^2_i/\mu^2 \simeq N+1$} indeed correspond to the correct
number of iterations required by SGD.

We therefore see that the choice between a dependence on the average
{\em quadratic} conditioning $\E L^2_i/\mu^2$, or a linear dependence
on the {\em uniform} conditioning $\sup L/\mu$, is unavoidable.  A
linear dependence on the average conditioning $\overline{L}/\mu$ is
not possible with any method that samples from the source distribution
$\DD$.  In the next Section, we will show how we {\em can} obtain a
linear dependence on the average conditioning $\overline{L}/\mu$,
using {\em importance sampling}, i.e.~by sampling from a modified
distribution.

\section{Importance Sampling}\label{sec:importance}

We will now consider stochastic gradient descent, where gradient estimates are sampled from a {\em
  weighted distribution}.  

\subsection{Reweighting a Distribution} 

For a weight function $w(i)$ which assigns a non-negative weight
$w(i)\geq 0$ to each index $i$, the weighted distribution $\DD^{(w)}$
is defined as the distribution such that
$$\mathbb{P}_{\DD^{(w)}}\left(I \right) \propto
\E_{i\!\sim\DD}\left[{\mathbbm{1}_I(i)w(i)}\right],$$ where $I$ is an
event (subset of indices) and $\mathbbm{1}_I(\cdot)$ its indicator
function.  For a discrete distribution $\DD$ with probability mass
function $p(i)$ this corresponds to weighting the probabilities to
obtain a new probability mass function:
$$p^{(w)}(i) \propto w(i) p(i).$$  Similarly, for a continuous
distribution, this corresponds to multiplying the density by $w(i)$
and renormalizing.

One way to construct the weighted distribution $\DD^{(w)}$, and sample
from it, is through \emph{rejection sampling}: sample $i\sim\DD$, and accept 
with probability $w(i)/W$, for some $W\geq \sup_i w(i)$.  Otherwise,
reject and continue to re-sample until a suggestion $i$ is
accepted.  The accepted samples are then distributed according to
$\DD^{(w)}$.

We use $\E^{(w)}[\cdot]=E_{i\sim\DD^{(w)}}[\cdot]$ to denote an
expectation where indices are sampled from the weighted distribution
$\DD^{(w)}$.  An important property of such an expectation is that for
any quantity $X(i)$ that depends on $i$:
\begin{equation}
  \label{eq:Ew}
  \E^{(w)}\left[\tfrac{1}{w(i)}X(i)\right] = \E\left[X(i)\right] /
  \E\left[w(i)\right],
\end{equation}
where recall that the expectations on the r.h.s.~are with respect to
$i\sim\DD$.  In particular, when $\E[w(i)]=1$, we have that
$\E^{(w)}\left[\frac{1}{w(i)}X(i)\right] = \E X(i)$.  In fact, we will
consider only weights s.t.~$\E[w(i)]=1$, and refer to such weights as
{\em normalized}.

\subsection{Reweighted SGD}

For any normalized weight function $w(i)$, we can weight each
component $f_i$, defining:
\begin{equation}
  \label{eq:fwi}
  \fwi(\x) = \frac{1}{w(i)}f_i(\x)
\end{equation}
and obtain
\begin{equation}
  \label{eq:Fbyfw}
  F(\x) = \E^{(w)}[\fwi(\x)].
\end{equation}
The representation \eqref{eq:Fbyfw} is an equivalent, and equally
valid, stochastic representation of the objective $F(\x)$, and we can
just as well base SGD on this representation.  In this case, at each
iteration we sample $i \sim \DD^{(w)}$ and then use $\nabla \fwi(\x) =
\frac{1}{w(i)}\nabla f_i(\x)$ as an unbiased gradient estimate.  SGD
iterates based on the representation \eqref{eq:Fbyfw}, which we will
also refer to as $w$-weighted SGD, are then given by
\begin{equation}
  \label{eq:SGDw}
  \x_{k+1} \leftarrow \x_k - \frac{\gamma}{w(i_k)} \nabla f_{i_k}(\x_k)
\end{equation}
where $\{i_k\}$ are drawn i.i.d.~from $\DD^{(w)}$.  

The important observation here is that all SGD guarantees are equally
valid for the $w$-weighted updates \eqref{eq:SGDw}--the objective is
the same objective $F(\x)$, the sub-optimality is the same, and the
minimizer $\xmin$ is the same.  We do need, however, to calculate the
relevant quantities controlling SGD convergence with respect to the
modified components $\fwi$ and the weighted distribution $\DD^{(w)}$.

\subsection{Strongly Convex Smooth Optimization using Weighted SGD}\label{sec:weightedSGD}

We now return to the analysis of strongly convex smooth optimization 
and investigate how re-weighting can yield a better
guarantee.  To do so, we must analyze the relevant quantities
involved.

The Lipschitz constant $\Lwi$ of each component $\fwi$ is now scaled,
and we have, $\Lwi = \frac{1}{w(i)} L_i$.  The supremum is given
by:
\begin{equation}
  \label{eq:supLw}
  \supLw = \sup_i \Lwi = \sup_i \frac{L_i}{w(i)}.
\end{equation}
It is easy to verify that \eqref{eq:supLw} is minimized by the weights
\begin{equation}
  \label{eq:wL}
  \wL(i) = \frac{L_i}{\overline{L}},
\end{equation}
and that with this choice of weights 
\begin{equation}
  \label{eq:optsupLw}
  \supLwL = \sup_i \frac{L_i}{L_i/\overline{L}} = \overline{L}.
\end{equation}
Note that the average Lipschitz constant
$\overline{L}=\E[L_i]=\E^{(w)}[\Lwi]$ is invariant under weightings.

Before applying Corollary \ref{thm:SGDsteps}, we must also calculate:
\begin{align}
  \sigmaLwL^2 &= \E^{(w)}[ \enormsq{\nabla \fwi(\xmin)} ] = \E^{(w)}[
  \frac{1}{w(i)^2} \enormsq{\nabla
    f_i(\xmin)}] \label{eq:sigmaLwL} \\
  &= \E[ \frac{1}{w(i)} \enormsq{\nabla f_i(\xmin)}] = \E[
  \frac{\overline{L}}{L_i} \enormsq{\nabla f_i(\xmin)}] \leq
  \frac{\overline{L}}{\inf L} \sigma^2. \notag  
\end{align}

Now, applying Corollary \ref{thm:SGDsteps} to the $\wL$-weighted SGD
iterates \eqref{eq:SGDw} with weights \eqref{eq:wL}, we have that,
with an appropriate stepsize,
{
\begin{align}
  k &= 2\log(2\varepsilon_0/\varepsilon) \Big(\frac{\supLwL}{\mu} +
  \frac{\sigmaLwL^2}{\mu^2 \varepsilon} \Big)  \label{eq:wLk} \\
&\leq 2\log(2\varepsilon_0/\varepsilon) \Big(\frac{\overline{L}}{\mu} +
\frac{\overline{L}}{\inf L} \cdot \frac{\sigma^2}{\mu^2 \varepsilon} \Big)  \notag
\end{align}
}
iterations are sufficient for $\E^{(w)} \enormsq{\xcur - \xmin} \leq
\varepsilon$, where $\xmin,\mu$ and $\varepsilon_0$ are exactly as in
Corollary \ref{thm:SGDsteps}.

\subsection{Partially biased sampling}

If $\sigma^2=0$, i.e.~we are in the ``realizable'' situation, with
true linear convergence, then we also have $\sigmaLwL^2=0$.  In this
case, we already obtain the desired guarantee: linear convergence with
a linear dependence on the average conditioning $\overline{L}/\mu$,
strictly improving over Bach and Moulines.  However, the inequality in
\eqref{eq:sigmaLwL} might be tight in the presence of components with
very small $L_i$ that contribute towards the residual error (as might
well be the case for a small component).  When $\sigma^2>0$, we
therefore get a dissatisfying scaling of the second term, relative to
Bach and Moulines, by a factor of $\overline{L}/{\inf L}$.

Fortunately, we can easily overcome this factor.  To do so, consider
sampling from a distribution which is a \emph{mixture} of the original source
distribution and its re-weighting using the weights \eqref{eq:wL}.
That is, sampling using the weights:
\begin{equation}
  \label{eq:wLmixed}
  \wL(i) = \frac{1}{2} + \frac{1}{2} \cdot \frac{L_i}{\overline{L}}.
\end{equation}
We refer to this as {\em partially biased sampling}.  Using these weights, we have
\begin{equation}
  \label{eq:supLwLmix}
  \supLwL = \sup_i \frac{1}{\frac{1}{2} + \frac{1}{2} \cdot
    \frac{L_i}{\overline{L}}} L_i \leq 2 \overline{L}
\end{equation}
and
\begin{equation}
  \label{eq:sigmaLwLmix}
  \sigmaLwL^2 = \E[ \frac{1}{\frac{1}{2} + \frac{1}{2} \cdot
    \frac{L_i}{\overline{L}}} \enormsq{\nabla f_i(\xmin)}] \leq 2
  \sigma^2.
\end{equation}
Plugging these into Corollary \ref{thm:SGDsteps} we obtain:
\begin{corollary}\label{thm:mixsteps} Let each $f_i$ be convex where
  $\nabla f_i$ has Lipschitz constant $L_i$ and let
  $\F(\x)=\E_{i\sim\DD}[f_i(\x)]$, where $\F(\x)$ is $\mu$-strongly
  convex.  Set $\sigma^2 = \E \enormsq{\nabla f_i(\xmin)}$, where $\xmin
  = \argmin_{\x} \F(\x)$. For any desired $\varepsilon$,
  using a stepsize of
$$ \gamma = \frac{\mu \varepsilon}{4(\varepsilon \mu \overline{L} + \sigma^2)}$$
we have that after
{
\begin{equation}
k = 4\log(2\varepsilon_0/\varepsilon) \Big(\frac{\overline{L}}{\mu} + \frac{\sigma^2}{\mu^2 \varepsilon} \Big) 
\end{equation}
}
iterations of $w$-weighted SGD \eqref{eq:SGDw} with weights specified
by \eqref{eq:wLmixed}, $\E^{(w)}\enormsq{\xcur - \xmin} \leq \varepsilon$,
where $\varepsilon_0 = \enormsq{\x_0 - \xmin}$ and $\overline{L}=\E L_i$.
\end{corollary}
We now obtain the desired linear scaling on $\overline{L}/\mu$,
without introducing any additional factor to the residual term, except
for a constant factor of two.  We thus obtain a result which dominates
Bach and Moulines (up to a factor of 2) and substantially improves
upon it (with a linear rather than quadratic dependence on the
conditioning).

One might also ask whether the previous best known result
\eqref{eq:bmbound} could be improved using weighted sampling.  The
relevant quantity to consider is the average square Lipschitz constant
for the weighted representation: \eqref{eq:Fbyfw}:
\begin{align}
  \sqLw \doteq &\E^{(w)}\left[\left(\Lwi\right)^2\right] = \E^{(w)}[
    \frac{L^2i}{w(i)^2}] = \E[\frac{L_i^2}{w(i)}]. 
\end{align}
Interestingly, this quantity is minimized by the same weights as
$\supLw$, given by \eqref{eq:wL}, and with these weights we have:
\begin{align}
  \sqLw &= \E[\frac{L_i^2}{L_i/\overline{L}}] = \overline{L}\E L_i = \overline{L}^2. 
\end{align}
Again, we can use the partially biased weights give in \eqref{eq:wLmixed}, which
yields $\sqLw \leq 2 \overline{L}^2$ and also ensures $\sigmaLwL^2\leq
2 \sigma^2$.  In any case, we get a dependence on $\overline{L}^2 =
(\E L_i)^2 \leq \E[L_i^2]$ instead of $\overline{L^2}=\E[L_i^2]$,
which is indeed an improvement.  Thus, the Bach and Moulines guarantee
is also improved by using biased sampling, and in particular the
partially biased sampling specified by the weights \eqref{eq:wLmixed}.
However, relying on Bach and Moulines we still have a quadratic
dependence on $(\overline{L}/\mu)^2$, as opposed to the linear
dependence we obtain in Corollary \ref{thm:mixsteps}.

\subsection{Implementing Importance Sampling}

As discussed above, when the magnitudes of $L_i$ are highly variable,
importance sampling is necessarily in order to obtain a dependence on
the average, rather than worst-case, conditioning.  In some applications,
especially when the Lipschitz constants are known in advance or easily
calculated or bounded, such importance sampling might be possible by
directly sampling from $D^{(w)}$.  This is the case, for example, in
trigonometric approximation problems or linear systems which need to
be solved repeatedly, or when the Lipschitz constant is easily
computed from the data, and multiple passes over the data are needed
anyway.  We do acknowledge that in other regimes, when data is
presented in an online fashion, or when we only have sampling access
to the source distribution $\DD$ (or the implied distribution over
gradient estimates), importance sampling might be difficult.

One option that could be considered, in light of the above results, is
to use rejection sampling to simulate sampling from $\DD^{(w)}$.
For the weights \eqref{eq:wL}, this can be done by accepting
samples with probability proportional to $L_i/\sup L$.  The overall
probability of accepting a sample is then $\overline{L}/\sup L$,
introducing an additional factor of $\sup L/\overline{L}$.  This
results in a sample complexity with a linear dependence on $\sup L$,
as in Corollary \ref{thm:SGDsteps}
(for the weights \eqref{eq:wLmixed}, we can first accept with
probability 1/2, and then if we do not accept, perform this
procedure).  Thus, if we are presented samples from $\DD$, and the
cost of obtaining the sample dominates the cost of taking the gradient
step, we do not gain (but do not lose much either) from rejection
sampling.  We might still gain from rejection sampling if the cost of
operating on a sample (calculating the actual gradient and taking a
step according to it) dominates the cost of obtaining it and (a bound
on) the Lipschitz constant.

\subsection{A Family of Partially Biased Schemes}

The choice of weights \eqref{eq:wLmixed} corresponds to an equal mix of uniform and fully biased sampling.   More generally, we could consider sampling according to any one of a family of weights which interpolate between uniform and fully biased sampling:
\begin{equation}
  \label{eq:partialbias}
  \wLA(i) = \lambda + (1-\lambda) \frac{L_i}{\overline{L}}, \quad \quad \lambda \in [0,1]. 
\end{equation}

To be concrete, we summarize below the a template algorithm for SGD with partially biased sampling:

\begin{center}
\begin{algorithm}[ht]
\caption{Stochastic Gradient Descent with Partially Biased Sampling}
	\label{alg:randomized-general}
\begin{center} \fbox{
\begin{minipage}{.95\textwidth} 
\vspace{4pt}
\alginout{\begin{itemize}
\item	Initial estimate $\xo \in \mathbb{R}^d$ 
\item Bias parameter $\lambda \in [0,1]$
\item Step size $\gamma>0$
\item	Tolerance parameter $\delta > 0$
\item Access to the source distribution $\DD$
\item If $\lambda<1$: bounds on the Lipschitz constants $L_i$; the
  weights $w^\lambda(i)$ derived from them (see eq. \ref{eq:partialbias});
  and access to the weighted distribution $\DD^{(\lambda)}$.
\end{itemize}}
{Estimated solution $\xhat$ to the problem $\min_{\x} \F(\x)$
}
\vspace{8pt}\hrule\vspace{8pt}

\begin{algtab*}
$k \leftarrow 0$

\algrepeat 
	$k \leftarrow k + 1$ \\
	Draw an index $i \sim \DD^{(\lambda)}$. \\
	$\x_k \leftarrow \x_{k-1} - \frac{\gamma}{w^\lambda(i)}\nabla f_i(\x_{k-1})$  \\ 
\alguntil{{$ \nabla\F(\x) \leq \delta$}} 
$\xhat \leftarrow \x_k$
\end{algtab*}
\end{minipage}}
\end{center}
\end{algorithm}
\end{center}
For arbitrary $\lambda \in [0,1]$, we have the bounds
\begin{equation}
  \supLwLA = \sup_i \frac{L_i}{\lambda+ (1-\lambda)
    \frac{L_i}{\overline{L}}} \leq  \min\left( \frac{\overline{L}}{1-\lambda}, \frac{\sup_i L_i}{\lambda} \right)
    \nonumber
\end{equation}
and
\begin{equation}
  \sigmaLwLA^2 = \E[ \frac{1}{\lambda + (1-\lambda)
    \frac{L_i}{\overline{L}}} \enormsq{\nabla f_i(\xmin)}] \leq \max \left( \frac{1}{\lambda}, \frac{\overline{L}}{(1-\lambda) \inf_i L_i} \right)\sigma^2
    \nonumber
\end{equation}
Plugging these quantities into Corollary \ref{thm:SGDsteps}, we obtain:
\begin{corollary}\label{thm:partial bias} Let each $f_i$ be convex where
  $\nabla f_i$ has Lipschitz constant $L_i$ and let
  $\F(\x)=\E_{i\sim\DD}[f_i(\x)]$, where $\F(\x)$ is $\mu$-strongly
  convex.  Set $\sigma^2 = \E \enormsq{\nabla f_i(\xmin)}$, where $\xmin
  = \argmin_{\x} \F(\x)$. For any desired $\varepsilon$,
  using a stepsize of
$$ \gamma = \frac{\mu \varepsilon}{2\varepsilon \mu  \min\left( \frac{\overline{L}}{1-\lambda}, \frac{\sup_i L_i}{\lambda} \right)+ 2\max \left( \frac{1}{\lambda}, \frac{\overline{L}}{(1-\lambda) \inf_i L_i} \right)\sigma^2}$$
we have that after
{
\begin{equation}
k = 2\log(2\varepsilon_0/\varepsilon) \left(\frac{ \min\left( \frac{\overline{L}}{1-\lambda}, \frac{\sup_i L_i}{\lambda} \right)}{\mu} + \frac{ \max \left( \frac{1}{\lambda}, \frac{\overline{L}}{(1-\lambda) \inf_i L_i} \right)\sigma^2
}{\mu^2 \varepsilon} \right) 
\nonumber
\end{equation}
}
iterations of $w$-weighted SGD \eqref{eq:SGDw} with partially biased weights
 \eqref{eq:partialbias}, $\E^{(w)}\enormsq{\xcur - \xmin} \leq \varepsilon$,
where $\varepsilon_0 = \enormsq{\x_0 - \xmin}$ and $\overline{L}=\E L_i$.
\end{corollary}
In this corollary, even if $\lambda$ is close to 1, i.e. we add only a
small amount of bias to the sampling, we obtain a bound with a linear
dependence on the average conditioning $\overline{L}/\mu$ (multiplied
by a factor of $\frac{1}{\lambda}$), since we can bound $\min\left(
  \frac{\overline{L}}{1-\lambda}, \frac{\sup_i L_i}{\lambda} \right)
\leq \frac{\overline{L}}{1-\lambda}$.

\section{Importance Sampling for SGD in Other Scenarios}\label{sec:other}

In the previous Section, we considered SGD for smooth and strongly
convex objectives, and were particularly interested in the regime
where the residual $\sigma^2$ is low, and the linear convergence term
is dominant.  Weighted SGD could of course be relevant also in other
scenarios, and we now briefly survey them, as well as relate them to
our main scenario of interest.

\subsection{Smooth, Not Strongly Convex}

When each component $f_i$ is convex, non-negative, and has an
$L_i$-Lipschitz gradient, but the objective $F(\x)$ is not necessarily
strongly convex, then after
\begin{equation}
  \label{eq:SST}
  k = O\left( \frac{ (\sup L) \enormsq{\xmin} }{\varepsilon} \cdot
    \frac{F(\xmin)+\varepsilon}{\varepsilon} \right)
\end{equation}
iterations of SGD with an appropriately chosen step-size we will have
$F(\overline{\xcur})\leq F(\xmin)+\varepsilon$, where
$\overline{\xcur}$ is an appropriate averaging of the $k$ iterates
\citet{SST2010}.  The relevant quantity here determining the iteration
complexity is again $\sup L$.  Furthermore, \citet{SST2010}, relying
on an example from \citet{FoygelSrebroCOLT2011}, point out that the
dependence on the supremum is unavoidable and {\em cannot} be replaced
with the average Lipschitz constant $\overline{L}$.  That is, if we
sample gradients according to the source distribution $\DD$, we must
have a linear dependence on $\sup L$.

The only quantity in the bound \eqref{eq:SST} that changes with a
re-weighting is $\sup L$---all other quantities ($\enormsq{\xmin}$,
$F(\xmin)$, and the sub-optimality $\varepsilon$) are invariant to
re-weightings.  We can therefor replace the dependence on $\sup L$
with a dependence on $\supLw$ by using a weighted SGD as in
\eqref{eq:SGDw}.  As we already calculated, the optimal weights are
given by \eqref{eq:wL}, and using them we have $\supLw =
\overline{L}$.  In this case, there is no need for partially biased
sampling, and we obtain that with an appropriate step-size,
\begin{equation}
  \label{eq:SST2}
  k = O\left( \frac{ \overline{L} \enormsq{\xmin} }{\varepsilon} \cdot
    \frac{F(\xmin)+\varepsilon}{\varepsilon} \right)
\end{equation}
iterations of weighed SGD updates \eqref{eq:SGDw} using the weights
\eqref{eq:wL} suffice.

We again see that using importance sampling allows us to reduce the
dependence on $\sup L$, which is unavoidable without biased sampling,
to a dependence on $\overline{L}$.

\subsection{Non-Smooth Objectives}

We now turn to non-smooth objectives, where the components $f_i$ might
not be smooth, but each component is $G_i$-Lipschitz.  Roughly
speaking, $G_i$ is a bound on the first derivative (gradient) of
$f_i$, while $L_i$ is a bound on the second derivatives of $f_i$.
Here, the performance of SGD depends on the second moment
$\overline{G^2}=\E[G_i^2]$.  The precise iteration
complexity depends on whether the objective is strongly convex or
whether $\xmin$ is bounded, but in either case depends linearly on
$\overline{G^2}$ (see e.g.~\cite{nes04,shamir2012stochastic}).

By using weighted SGD we can replace the linear dependence on
$\overline{G^2}$ with a linear dependence on $\sqGw = \E^{(w)}\left[
(\Fwi)^2 \right]$, where $\Fwi$ is the Lipschitz constant of the
scaled $\fwi$ and is given by $\Fwi = G_i/w(i)$.  Again, this follows
directly from the standard SGD guarantees, where we consider the
representation \eqref{eq:Fbyfw} and use any subgradient from
$\partial \fwi(\x)$.  

We can calculate:
\begin{equation}
  \label{eq:sqGw}
  \sqGw = \E^{(w)}\left[\frac{G^2_i}{w(i)^2}\right] = \E\left[\frac{G^2_i}{w(i)}\right] 
\end{equation}
which is minimized by the weights:
\begin{equation}\label{eq:wG}
  w(i) = \frac{G_i}{\overline{G}}
\end{equation}
where $\overline{G}=\E G_i$.  Using these weights we have $\sqGw =
\E[G_i]^2 = \overline{G}^2$.  Using importance sampling, we can thus
reduce the linear dependence on $\overline{G^2}$ to a linear
dependence on $\overline{G}^2$.  Its helpful to recall that
$\overline{G^2} = \overline{G}^2 + \variance[G_i]$.  What we save is
therefore exactly the variance of the Lipschitz constants
$G_i$.

In parallel work, \citet{zhao2014stochastic} also consider importance
sampling for stochastic optimization for non-smooth objectives.
\citeauthor{zhao2014stochastic} consider a more general setting,
with a composite objective that is only partially linearized.  But
also there, the iteration complexity depends on the second moment of
the gradient estimates, and the analysis performed above applies
(\citeauthor{zhao2014stochastic} perform a specialized analysis
instead).

\subsection{Non-Realizable Regime}

Returning to the smooth and strongly convex setting of Sections \ref{sec:SGD}
and \ref{sec:importance}, let us consider more carefully the residual term $\sigma^2
= \E \enormsq{\nabla f_i(\xmin)}$.  This quantity definitively depends
on the weighting, and in the analysis of Section \ref{sec:weightedSGD}, we avoided
increasing it too much, introducing partial biasing for this purpose.
However, if this is the dominant term, we might want to choose weights
so as to minimize this term.  The optimal weights here would be
proportional to $\enorm{\nabla f_i(\xmin)}$.  The problem is that we
do not know the minimizer $\xmin$, and so cannot calculate these
weights.  Approaches which dynamically update the weights based on the
current iterates as a surrogate for $\xmin$ are possible, but beyond
the scope of this paper.

An alternative approach is to bound $\enorm{\nabla f_i(\xmin)} \leq
G_i$ and so $\sigma^2 \leq \overline{G^2}$.  Taking this bound, we are
back to the same quantity as in the non-smooth case, and the optimal
weights are proportional to $G_i$.  Note that this is a different
weighting then using weights proportional to $L_i$, which optimize the
linear-convergence term as studied in Section \ref{sec:weightedSGD}.

To understand how weighting according to $G_i$ and $L_i$ are
different, consider a generalized linear objective where
$f_i(\x)=\phi_i(\ip{\z_i}{\x})$, and $\phi_i$ is a scalar function
with $|\phi'_i|\leq G_\phi$ and $|\phi''_i| \leq L_\phi$.  We have that
$G_i \propto \enorm{\z_i}$ while $L_i \propto \enormsq{\z_i}$.
Weighting according to the Lipschitz constants of the gradients,
i.e.~the ``smoothness'' parameters, as in \eqref{eq:wL}, versus
weighting according to the Lipschitz constants of $f_i$ as in
\eqref{eq:wG}, thus corresponds to weighting according to
$\enormsq{\z_i}$ versus $\enorm{\z_i}$, and are rather different.  We
can also calculate that weighing by $L_i \propto \enormsq{\z_i}$ (i.e.~following
\eqref{eq:wL}), yields $\sqGw = \overline{G^2} > \overline{G}^2$.
That is, weights proportional to $L_i$ yield a suboptimal
gradient-dependent term (the same dependence as if no weighting at all
was used).  Conversely, using weights proportional to $G_i$,
i.e.~proportional to $\enorm{\z_i}$ yields $\supLw =
(E[\sqrt{L_i}])\sqrt{\sup L}$ -- a suboptimal dependence, though better
then no weighting at all.

Again, as with partially biased sampling, we can weight by the average, $w(i) = \frac{1}{2}\cdot\frac{G_i}{\bar{G}}+\frac{1}{2}\cdot\frac{L_i}{\bar{L}}$ and ensure both terms are optimal up to a factor of two.

 \section{The least squares case and the Randomized Kaczmarz Method}\label{sec:kz}

 A special case of interest is the least squares problem, where 
\begin{equation}\label{eq:min}
\F(\x) = \frac{1}{2} \sum_{i=1}^n (\langle \ai, \x \rangle - b_i )^2  = \frac{1}{2} \|\A\x - \b\|_2^2
\end{equation}
with $\b$ an $n$-dimensional vector, $\A$ an $n\times d$ matrix with rows
$\ai$,  and $\xmin = \argmin_{\x}  \frac{1}{2} \|\A\x - \b\|_2^2$ is the
least-squares solution.  Writing the least squares problem
\eqref{eq:min} in the form \eqref{eq:min2}, we see that the source distribution $\DD$ is uniform over $\{1,2, \dots, n\}$, the components are $f_i=\frac{n}{2}(\langle \ai, \x \rangle - b_i )^2$, the Lipschitz constants are $L_i = n \| \ai \|_2^2$, the average Lipschitz constant is $\frac{1}{n} \sum_i L_i = \| \A \|_F^2$, the strong convexity parameter is $\mu = \frac{1}{\|(\A^T\A)^{-1}\|_2}$, so that $K(\A)  := \overline{L}/\mu = \| \A \|^2_F \| (\A^T\A)^{-1}\|_2$, and the residual is $\sigma^2 =  n\sum_i \enormsq{ \ai } |\ip{\ai}{\xmin}-\bi|^2$.  Note that in the case that $\A$ is not full-rank, one can instead replace $\mu$ with the smallest nonzero eigenvalue of $\A^*\A$ as in~\citep[Equation (3)]{liu2014asynchronous}.  In that case, we instead write $K(\A) = \| \A \|^2_F \| (\A^T\A)^{\dagger}\|_2$ as the appropriate condition number.

The randomized \emph{Kaczmarz method}~\citep{SV09:Randomized-Kaczmarz,CFMSS92:New-Variants,HM93:Algebraic-Reconstruction,Nat01:Mathematics-Computerized,whitney1967two,censor1983strong,tanabe1971projection,hanke1990acceleration,ZF12:Randomized-Extended,Nee10:Randomized-Kaczmarz} for solving the least squares problem \eqref{eq:min} begins with an arbitrary estimate $\x_0$, and in the $k$th iteration selects a row $i = i(k)$ i.i.d. at random from the matrix $\A$ and iterates by:
\begin{equation}\label{basic:iter}
\xnext = \xcur + c\cdot\frac{\bi - \<\ai, \xcur\>}{\|\ai\|_2^2}\ai,
\end{equation}
where the step size $c=1$ in the standard method.
 
Strohmer and Vershynin provided the first non-asymptotic convergence
rates, showing that drawing rows proportionally to $\enormsq{\ai}$
leads to provable exponential convergence in
expectation for the full-rank case~\citep{SV09:Randomized-Kaczmarz}.  Their method can easily be extended to the case when the matrix is not full-rank to yield convergence to some solution, see e.g.~\citep[Equation (3)]{liu2014asynchronous}. Recent works use
acceleration techniques to improve convergence
rates~\citep{lee2013efficient,Pop98:Extensions-Block-Projections,EN11:Acceleration-Randomized,popa2012kaczmarz,ZF12:Randomized-Extended,Elf80:Block-Iterative-Methods,EHL81:Iterative-Algorithms,Byr08:Applied-Iterative,Pop99:Block-Projections-Algorithms,Pop01:Fast-Kaczmarz-Kovarik,Pop04:Kaczmarz-Kovarik-Algorithm,NW12:Two-Subspace-Projection,needell2013paved}.

However, one can easily verify that the iterates \eqref{basic:iter}
are precisely weighted SGD iterates \eqref{eq:SGDw} with the fully
biased weights \eqref{eq:wL}.

The reduction of the quadratic dependence on the conditioning to a
linear dependence in Theorem \ref{thm:gen}, as well as the use of
biased sampling which we investigate here was motivated by
\citeauthor{SV09:Randomized-Kaczmarz}'s analysis of the randomized
Kaczmarz method.  Indeed, applying Theorem \ref{thm:gen} to the
weighted SGD iterates \eqref{eq:SGD} for \eqref{eq:min} with the
weights \eqref{eq:wL} and a stepsize of $\gamma=1$ yields precisely
the \citet{SV09:Randomized-Kaczmarz} guarantee.

Understanding the randomized Kaczmarz method as SGD allows us also
to obtain improved methods and results for the randomized Kaczmarz
method:

\bigskip

\paragraph{\bf Using Step-sizes.}
As shown by~\citet{SV09:Randomized-Kaczmarz} and extended
by~\citet{Nee10:Randomized-Kaczmarz}, the randomized Kaczmarz method
with weighted sampling exhibits exponential convergence, but only to
within a radius, or \emph{convergence horizon}, of the least-squares
solution.  This is because a step-size of $\gamma=1$ is used, and so
the second term in \eqref{recursion} does not vanish.  It has been
shown~\citep{whitney1967two,censor1983strong,tanabe1971projection,hanke1990acceleration,NW12:Two-Subspace-Projection}
that changing the step size can allow for convergence inside of this
convergence horizon, although non-asymptotic results have been
difficult to obtain.  Our results allow for finite-iteration
guarantees with arbitrary step-sizes and can be immediately applied to this
setting.  Indeed, applying Theorem \ref{thm:gen} with the weights \eqref{eq:wL}
 gives
\begin{corollary}
\label{thm:genSV}
Let $\A$ be an $n\times d$ matrix with rows $\ai$.  Set $\e = \A\xmin - \b$, where $\xmin$ is the minimizer of the problem 
$$
\xmin = \argmin_{\x} \frac{1}{2}\enormsq{\A\x-\b}.
$$
Suppose that $c < 1$.  
Set $a^2_{\min} = \inf_i\enormsq{\ai}$, $a^2_{\max} = \sup_i\enormsq{\ai}$.
Then the expected error at the $kth$ iteration of the Kaczmarz method described by~\eqref{basic:iter} with row $\ai$ selected with probability $p_i = \enormsq{\ai}/\|\A\|^2_F$  satisfies
\begin{equation}
\label{err:K}
\E\enormsq{\xcur - \xmin} \leq \left[ 1 - \frac{2c(1-c)}{K(\A)}  \right]^k\enormsq{\xo - \xmin} +  \frac{c}{1-c} K(\A) r,
\end{equation}
with $r = \frac{\sigma^2}{n \| \A \|_F^2 \cdot a^2_{\min}}$.  The expectation is taken with respect to the weighted distribution over the rows.
\end{corollary}
When e.g. $c=\frac{1}{2}$, we recover the exponential rate of~\citet{SV09:Randomized-Kaczmarz} up to a factor of $2$,  and nearly the same convergence horizon.  For arbitrary $c$, Corollary \ref{thm:genSV}  implies a tradeoff between a smaller convergence horizon and a slower convergence rate. 

\bigskip

\paragraph{\bf Uniform Row Selection.} The Kaczmarz variant of \citet{SV09:Randomized-Kaczmarz} calls for weighted row sampling, and thus requires
pre-computing all the row norms.  Although certainly possible in some
applications, in other cases this might be better avoided.
Understanding the randomized Kaczmarz as SGD allows us to apply
Theorem \ref{thm:gen} also with uniform weights (i.e.~to the
unbiased SGD), and obtain a randomized Kaczmarz using uniform
sampling, which converges to the least-squares solution and enjoys
finite-iteration guarantees:
\begin{corollary}\label{thm:genUnif}
Let $\A$ be an $n\times d$ matrix with rows $\ai$.   Let $\D$ be the diagonal matrix with terms $d_{j,j} = \| \ai \|_2$, and consider the composite matrix $\D^{-1} \A$.  Set $\e_{w} = \D^{-1} (\A \xmin^{w} - \b)$, where $\xmin^{w}$ is the minimizer of the weighted least squares problem 
\begin{equation}
\nonumber
\xmin^{w} = \argmin_{\x} \frac{1}{2}\enormsq{\D^{-1}(\A \x-\b)}.
\end{equation}
Suppose that $c < 1$.
Then the expected error after $k$ iterations of the Kaczmarz method described by~\eqref{basic:iter} with uniform row selection satisfies
\begin{equation}\nonumber
\E\enormsq{\xcur - \xmin^{w}} \leq \left[ 1 - \frac{2c(1-c)}{K(\D^{-1}\A)} \Big) \right]^k\enormsq{\xo - \xmin^{w}} + \frac{ c}{1-c} K(\D^{-1}\A) r_{w},
\end{equation}
where $r_{w} =  \enormsq{ \e_{w} }/n.$
\end{corollary}
Note that the randomized Kaczmarz algorithm with uniform row selection converges exponentially to a \emph{weighted} least-squares solution, to within arbitrary accuracy by choosing sufficiently small stepsize $c$. Thus, in general, the randomized Kaczmarz algorithms with uniform and biased row selection converge (up to a convergence horizon) towards different solutions.     

\bigskip

\paragraph{\bf Partially Biased Sampling.}
As in our SGD analysis, using the partially biased sampling weights is applicable also for the randomized Kaczmarz
method.  Applying Theorem \ref{thm:gen} using weights \eqref{eq:wLmixed} gives
\begin{corollary}[Randomized Kaczmarz with partially biased sampling]\label{thm:genHybrid}
Let $\A$ be an $n\times d$ matrix with rows $\ai$.   Set $\e = \A\xmin - \b$, where $\xmin$ is the minimizer of the problem 
$$
\xmin = \argmin_{\x} \frac{1}{2}\enormsq{\A\x-\b}.
$$
Suppose $c < 1/2$.  Then the iterate $\xcur$ of the modified Kaczmarz method described by
\begin{equation}
\label{hybridK}
\xnext = \xcur + 2c\cdot\frac{\bi - \<\ai, \xcur\>}{\| \A \|_F^2/n + \|\ai\|_2^2}\ai
\end{equation}
with row $\ai$ selected with probability $p_i = \frac{1}{2} \cdot \frac{\|\ai\|_2^2}{\| \A \|_F^2} + \frac{1}{2} \cdot \frac{1}{n}$ satisfies
\begin{equation}
\label{err:Khybrid}
\E\enormsq{\xcur - \xmin} \leq \left[ 1 - \frac{2c (1-2 c) }{K(\A)}\right]^k\enormsq{\xo - \xmin} + \frac{c K(\A)}{1-2 c} \cdot \frac{2 \sigma^2}{n \| \A \|_F^2}
\end{equation}
\end{corollary}
The partially biased randomized Kaczmarz method described above (which
does have modified update equation \eqref{hybridK} compared to the
standard update equation \eqref{basic:iter}) yields the same
convergence rate as the fully biased randomized Kaczmarz method
\citep{SV09:Randomized-Kaczmarz} (up to a factor of 2), but gives a
better dependence on the residual error over the fully biased
sampling, as the final term in \eqref{err:Khybrid} is smaller than the
final term in \eqref{err:K}.

\section{Numerical Experiments}\label{sec:exp}

In this section we present some numerical results for the randomized Kaczmarz algorithm with partially biased sampling,  that is, applying Algorithm \ref{alg:randomized-general} to the least squares problem $\F(\x) = \frac{1}{2} \|\A\x - \b\|_2^2$ (so $f_i(\x) = \frac{n}{2} ( \ip{\ai}{\x} - \bi )^2$) and considering $\lambda \in [0,1]$.  Recall that $\lambda = 0$ corresponds to the randomized Kaczmarz algorithm of Strohmer and Vershynin with fully weighted sampling \citep{SV09:Randomized-Kaczmarz}.  $\lambda = .5$ corresponds to the partially biased randomized Kaczmarz algorithm outlined in Corollary \ref{thm:genHybrid}.  We demonstrate how the behavior of the algorithm depends on $\lambda$, the conditioning of the system, and the residual error at the least squares solution. We focus on exploring the role of $\lambda$ on the convergence rate of the algorithm for various types of matrices $\A$.  We consider five types of systems, described below, each using a $1000\times 10$ matrix $\A$.  In each setting, we create a vector $\x$ with standard normal entries.  For the described matrix $\A$ and residual $\e$, we create the system $\b = \A\x + \e$ and run the randomized Kaczmarz method with various choices of $\lambda$.  Each experiment consists of $100$ independent trials and uses the optimal step size as in Corollary~\ref{thm:partial bias} with $\varepsilon = .1$; the plots show the average behavior over these trials. The settings below show the various types of behavior the Kazcmarz method can exhibit.  

\begin{description}
\item[Case 1] Each row of the matrix $\A$ has standard normal entries, except the last row which has normal entries with mean $0$ and variance $10^2$.  The residual vector $\e$ has normal entries with mean $0$ and variance $0.1^2$.

\item[Case 2] Each row of the matrix $\A$ has standard normal entries.  The residual vector $\e$ has normal entries with mean $0$ and variance $0.1^2$.  

\item[Case 3] The $j$th row of $\A$ has normal entries with mean $0$ and variance $j$.  The residual vector $\e$ has normal entries with mean $0$ and variance $20^2$.

\item[Case 4] The $j$th row of $\A$ has normal entries with mean $0$ and variance $j$.  The residual vector $\e$ has normal entries with mean $0$ and variance $10^2$.

\item[Case 5] The $j$th row of $\A$ has normal entries with mean $0$ and variance $j$.  The residual vector $\e$ has normal entries with mean $0$ and variance $0.1^2$.

\end{description}

\begin{figure}
\begin{tabular}{ccc}
{\bfseries Case 1:} & {\bfseries Case 2:}\\
\includegraphics[width=2.5in]{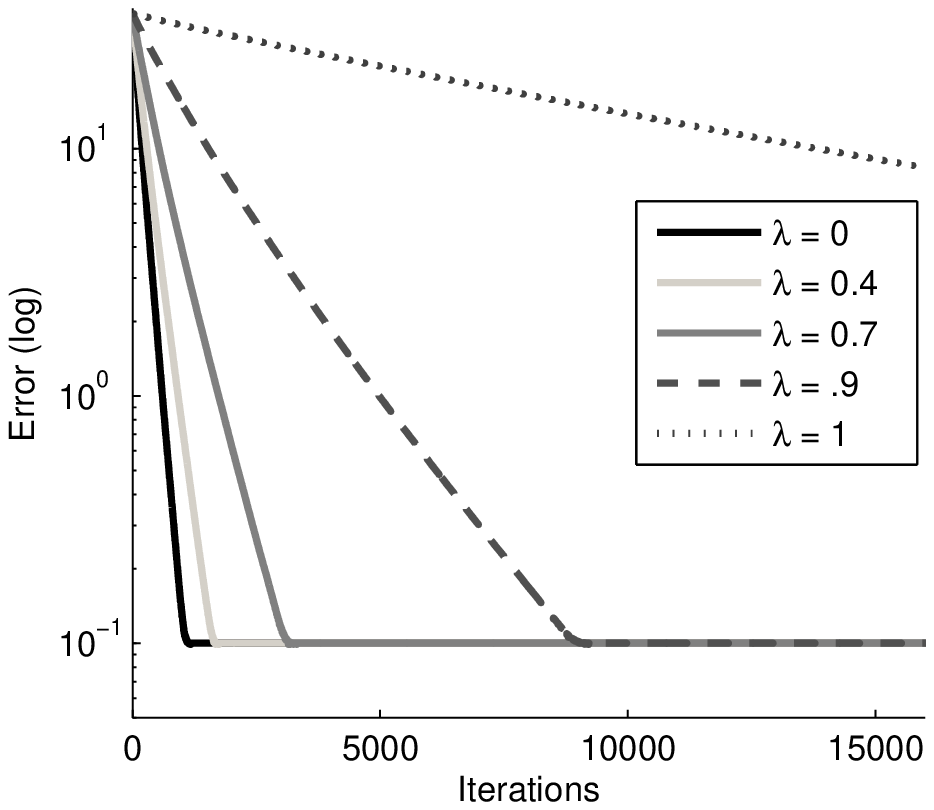} & \includegraphics[width=2.5in]{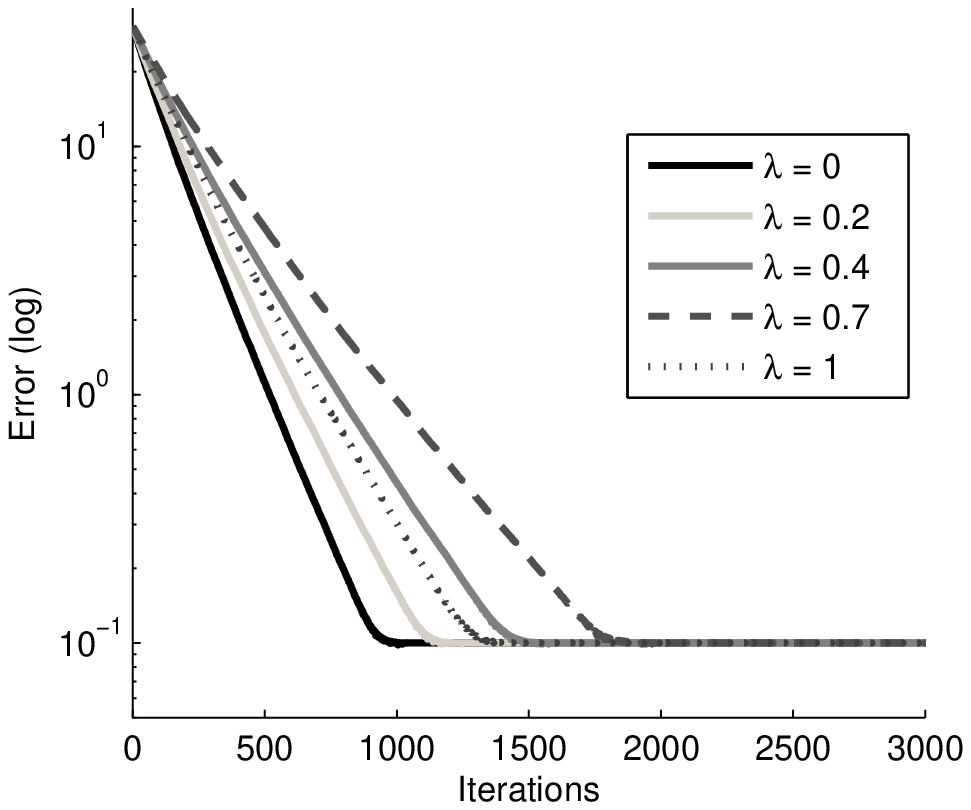} \\
\\
{\bfseries Case 3:} & {\bfseries Case 4:}\\
\includegraphics[width=2.5in]{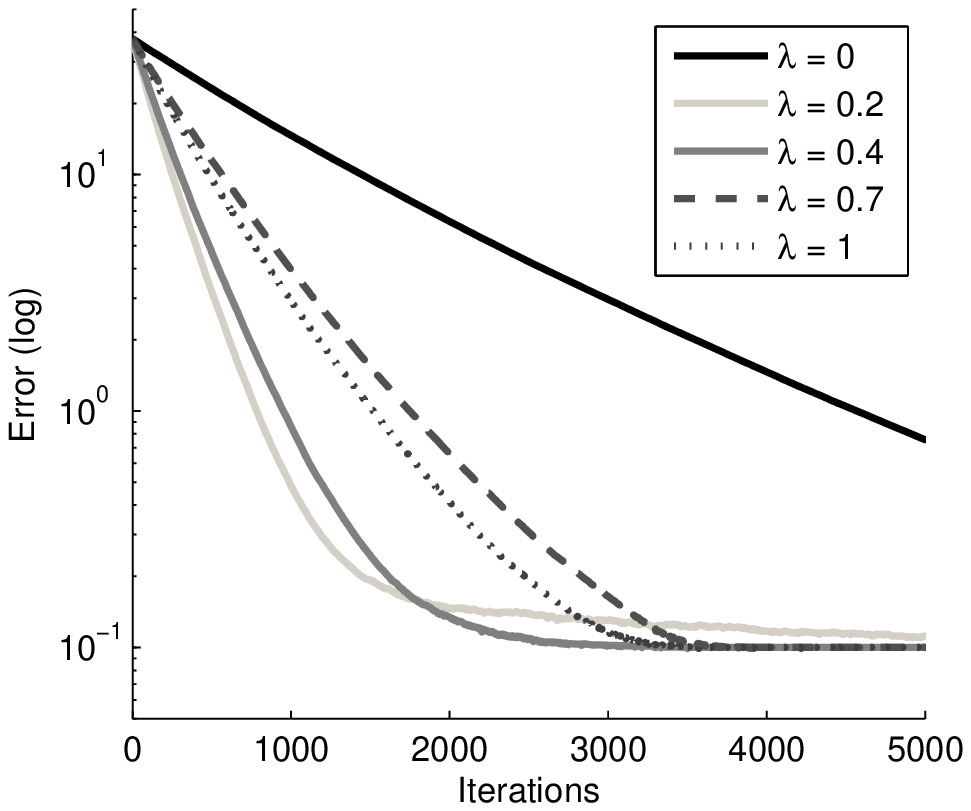} & \includegraphics[width=2.5in]{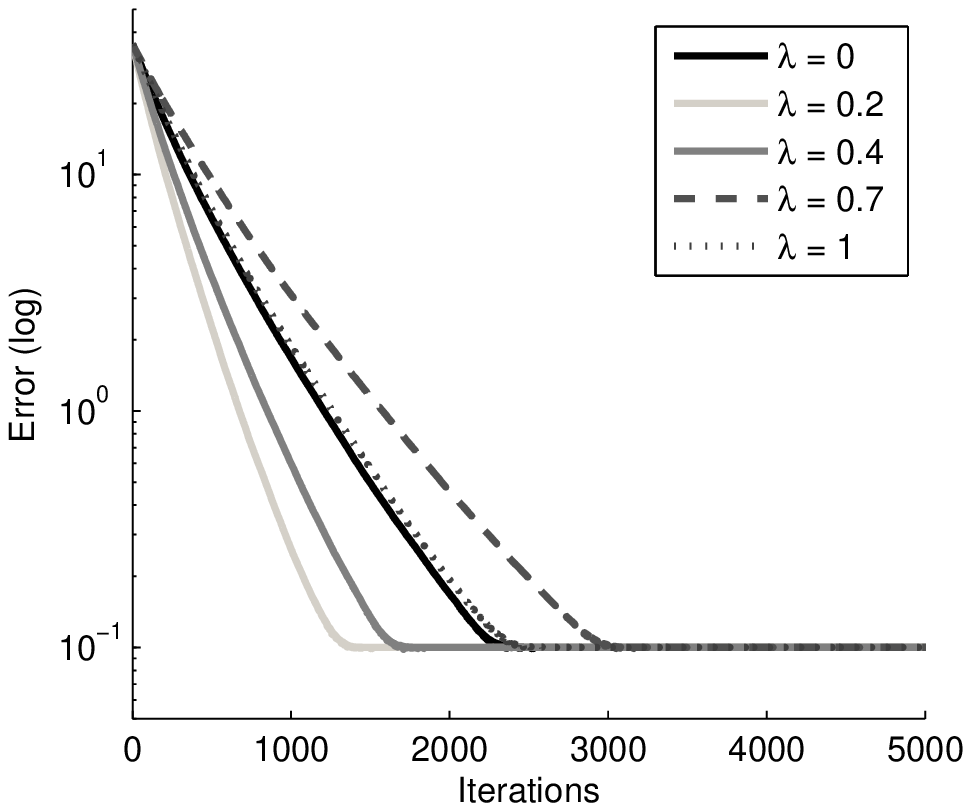} \\
\\
{\bfseries Case 5:} \\
\includegraphics[width=2.5in]{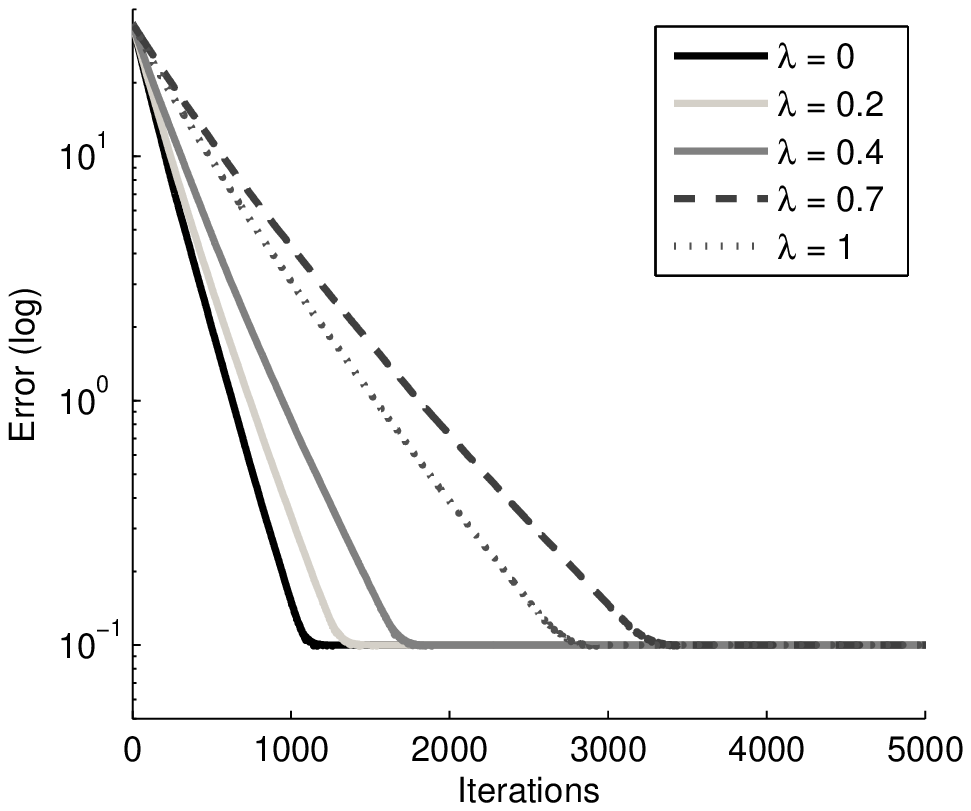} \\
\end{tabular}
\caption{The convergence rates for the randomized Kaczmarz method with various choices of $\lambda$ in the five settings described above.  The vertical axis is in logarithmic scale and depicts the approximation error $\enormsq{\xcur - \xmin}$ at iteration $k$ (the horizontal axis).}\label{fig1}
\end{figure}

Figure~\ref{fig1} shows the convergence behavior of the randomized Kaczmarz method in each of these five settings.  As expected, when the rows of $\A$ are far from normalized, as in Case 1, we see different behavior as $\lambda$ varies from $0$ to $1$.  Here, weighted sampling ($\lambda = 0$) significantly outperforms uniform sampling ($\lambda = 1$), and the trend is monotonic in $\lambda$.  On the other hand, when the rows of $\A$ are close to normalized, as in Case 2, the various $\lambda$ give rise to similar convergence rates, as is expected.   Out of the $\lambda$ tested (we tested increments of $0.1$ from $0$ to $1$), the choice $\lambda = 0.7$ gave the worst convergence rate, and again purely weighted sampling gives the best.  Still, the worst-case convergence rate was not much worse, as opposed to the situation with uniform sampling in Case 1.   Cases 3, 4, and 5 use matrices with  varying row norms and cover ``high", ``medium", and ``low" noise regimes, respectively.    In the high noise regime (Case 3), we find that fully weighted sampling, $\lambda = 0$, is relatively very slow to converge, as the theory suggests, and hybrid sampling outperforms both weighted and uniform selection.   In the medium noise regime (Case 4), hybrid sampling still outperforms both weighted and uniform selection.  Again, this is not surprising, since hybrid sampling allows a balance between small convergence horizon (important with large residual norm) and convergence rate. As we decrease the noise level (as in Case 5), we see that again weighted sampling is preferred. 

\begin{figure}
\begin{center}
\includegraphics[width=3.2in]{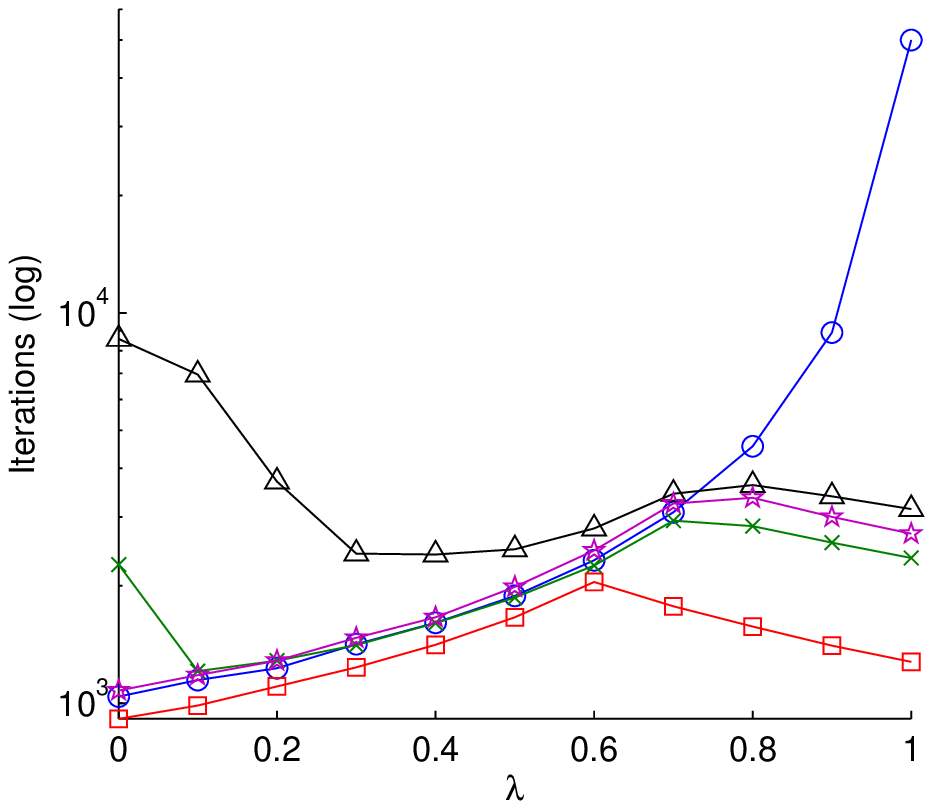}
\end{center}
\caption{Number of iterations $k$ needed by the randomized Kaczmarz method with partially biased sampling, for various values of $\lambda$, to obtain approximation error $\enormsq{\xcur - \xmin} \leq \varepsilon = 0.1$ in the five cases described above: Case 1 (blue with circle marker), Case 2 (red with square marker), Case 3 (black with triangle marker), Case 4 (green with x marker), and Case 5 (purple with star marker).  
}\label{fig2}
\end{figure}

Figure~\ref{fig2} shows the number of iterations of the randomized Kaczmarz method needed to obtain a fixed approximation error.  For the choice $\lambda = 1$ for Case 1, we cut off the number of iterations after 50,000, at which point the desired approximation error was still not attained.  As seen also from Figure~\ref{fig1}, Case 1 exhibits monotonic improvements as we scale $\lambda$.  For Cases 2 and 5, the optimal choice is pure weighted sampling, whereas Cases 3 and 4 prefer intermediate values of $\lambda$.

\section{Summary and outlook}

We consider this paper as making three contributions: the improved
dependence on the conditioning for smooth and strongly convex SGD, the
discussion of importance sampling for SGD, and the connection between
SGD and the randomized Kaczmarz method.

For simplicity, we only considered SGD iterates with a fixed step-size $\gamma$. This is enough for getting the optimal iteration complexity if the target accuracy $\varepsilon$ is known in advance, which was our approach in this paper. It is easy to adapt the analysis, using standard techniques, to incorporate decaying step-sizes, which are appropriate if we don't know $\varepsilon$ in advance.

We suspect that the assumption of strong convexity
can be weakened to \emph{restricted strong convexity} \citep{ly13, zy13} without changing any of the results of this paper; we leave this analysis to future work.

Finally, our discussion of importance sampling is limited to a static
reweighting of the sampling distribution.  A more sophisticated
approach would be to update the sampling distribution dynamically as
the method progresses, and as we gain more information about the relative
importance of components.  Although such dynamic importance sampling
is sometimes attempted heuristically, we are not aware of any rigorous
analysis of such a dynamic biasing scheme.

\section*{Acknowledgements}
We would like to thank the anonymous reviewers for their useful feedback which significantly improved the manuscript.  We would like to thank Chris White for pointing out a simplified proof of Corollary \ref{thm:SGDsteps}.
DN was partially supported by a Simons Foundation Collaboration
  grant, NSF CAREER $\#1348721$ and an Alfred P. Sloan Fellowship. NS was partially supported by a Google Research Award.
	RW was supported in part by ONR Grant N00014-12-1-0743, an AFOSR
  Young Investigator Program Award, and an NSF CAREER award.

\bibliographystyle{plainnat}
\bibliography{rk}

\appendix
\section{Proofs}\label{sec:proofs}
Our main results utilize an elementary fact about smooth functions with Lipschitz continuous gradient, called the co-coercivity of the gradient.  We state the lemma and recall its proof for completeness.
\subsection{The Co-coercivity Lemma}
\begin{lemma}[Co-coercivity]
\label{cocoercivity}
For a smooth function $f$ whose gradient has Lipschitz constant $L$,
$$
\enormsq{\nabla f(\x) - \nabla f(\y)} \leq L\<\x-\y, \nabla f(\x) - \nabla f(\y)\>.
$$
\end{lemma}
\begin{proof}
Since $\nabla f$ has Lipschitz constant $L$, if $\xmin$ is the minimizer of $f$, then \citep[see e.g.][page 26]{nes04}
\begin{equation}\label{eq:Lmin}
 \frac{1}{2L}\enormsq{\nabla f(\x)-\nabla f(\xmin)} =  \frac{1}{2L}\enormsq{\nabla f(\x)-\nabla f(\xmin)}  + \< \x - \xmin, \nabla f(\xmin) \>  \leq f(\x) - f(\xmin);
\end{equation}
Now define the convex functions
$$
G(\z) = f(\z) - \<\nabla f(\x), \z\>, \quad\text{and}\quad H(\z) = f(\z) - \<\nabla f(\y), \z\>,
$$
and observe that both have Lipschitz constants $L$ and minimizers $\x$ and $\y$, respectively.  Applying \eqref{eq:Lmin} to these functions therefore gives that
$$
G(\x) \leq G(\y) - \frac{1}{2L}\enormsq{\nabla G(\y)}, \quad\text{and}\quad H(\y) \leq H(\x) - \frac{1}{2L}\enormsq{\nabla H(\y)}.
$$
By their definitions, this implies that
\begin{align*}
f(\x) - \<\nabla f(\x),\x\> &\leq f(\y) - \<\nabla f(\x), \y\> - \frac{1}{2L}\enormsq{\nabla f(\y) - \nabla f(\x)} \\
\vspace{1cm}
f(\y) - \<\nabla f(\y),\y\> &\leq f(\x) - \<\nabla f(\y), \x\> - \frac{1}{2L}\enormsq{\nabla f(\x) - \nabla f(\y)}.
\end{align*}

Adding these two inequalities and canceling terms yields the desired result.

\end{proof}

\subsection{Proof of Theorem~\ref{thm:gen}}
With the notation of Theorem~\ref{thm:gen}, and where $i$ is the
random index chosen at iteration $k$, and $w=w_{\lambda}$, we have
\begin{align*}
\enormsq{\xnext -\xmin} &= \enormsq{\xcur - \xmin - \gamma\nabla f_i(\xcur)}\\
&= \enormsq{(\xcur - \xmin) -  \gamma(\nabla f_i(\xcur) - \nabla f_i(\xmin)) - \gamma\nabla f_i(\xmin)}\\
&= \enormsq{\xcur - \xmin} - 2 \gamma\<\xcur-\xmin, \nabla f_i(\xcur)\> + \gamma^2\enormsq{\nabla f_i(\xcur) - \nabla f_i(\xmin) + \nabla f_i(\xmin)}\\
&\leq \enormsq{\xcur - \xmin} - 2 \gamma\<\xcur-\xmin, \nabla f_i(\xcur)\> + 2\gamma^2\enormsq{\nabla f_i(\xcur) - \nabla f_i(\xmin)} + 2\gamma^2\enormsq{\nabla f_i(\xmin)} \\
&\leq \enormsq{\xcur - \xmin} - 2 \gamma\<\xcur-\xmin, \nabla f_i(\xcur)\> \\
& \hspace{5mm} + 2\gamma^2 L_i \<\xcur-\xmin, \nabla f_i(\xcur) - \nabla f_i(\xmin)\> + 2\gamma^2\enormsq{\nabla f_i(\xmin)},
\end{align*}
where we have employed Jensen's inequality in the first inequality and
the co-coercivity Lemma \ref{cocoercivity} in the final line.  We next
take an expectation with respect to the choice of $i$.  By assumption, $i\sim\DD$ such that $F(\x) = \mathbb{E} f_i(\x)$ and $\sigma^2 = \E \| \nabla f_i(\xmin)\|^2$.   
Then $\E\nabla f_i(\x)=\Fradg(\x)$, and we obtain:
\begin{align*}
\E\enormsq{\xnext -\xmin} &\leq \enormsq{\xcur - \xmin} 
- 2\gamma \<\xcur-\xmin, \Fradg(\xcur)\> 
+ 2\gamma^2 \E \left[ L_i
\<\xcur-\xmin, \nabla f_i(\xcur) - \nabla f_i(\xmin)\>\right] \nonumber\\
&\quad\quad+ 2\gamma^2 \E\enormsq{\nabla f_i(\xmin)} \\
&\leq \enormsq{\xcur - \xmin} 
- 2\gamma \<\xcur-\xmin, \Fradg(\xcur)\> 
+ 2\gamma^2 \sup_i L_i \E \<\xcur-\xmin, \nabla f_i(\xcur) - \nabla f_i(\xmin)\> \nonumber \\ 
&\quad\quad+ 2\gamma^2  \E\enormsq{\nabla f_i(\xmin)} \\
&=  \enormsq{\xcur - \xmin} - 2\gamma \<\xcur-\xmin, \Fradg(\xcur)\> +
2\gamma^2 \sup L  \<\xcur-\xmin, \Fradg(\xcur) - \Fradg(\xmin)\> +
2\gamma^2  \sigma^2\\
\intertext{
  We now utilize the strong convexity of $\F(\x)$ and obtain that}
&\leq \enormsq{\xcur - \xmin} - 2\gamma \mu(1 - \gamma \sup L) \enormsq{\xcur-\xmin} + 2\gamma^2 \sigma^2 \\
&= (1 - 2\gamma \mu(1 - \gamma \sup L)) \enormsq{\xcur-\xmin} + 2\gamma^2 \sigma^2 \\
\end{align*}
 when $\gamma \leq \frac{1}{ \sup L}$.
Recursively applying this bound over the first $k$ iterations yields the desired result,
\begin{align*}
\E\enormsq{\xcur -\xmin} &\leq \Big( 1 - 2\gamma \mu(1 - \gamma \sup L) \Big) \Big)^k\enormsq{\xo - \xmin} + 2\sum_{j=0}^{k-1} \Big( 1 - 2\gamma \mu(1 - \gamma \sup L) \Big) \Big)^j \gamma^2  \sigma^2 \\
&\leq \Big( 1 - 2\gamma \mu(1 - \gamma \sup L) \Big) \Big)^k\enormsq{\xo - \xmin} + \frac{\gamma  \sigma^2}{\mu\big( 1 - \gamma \sup L \big)}.
\end{align*}

\appendix

\end{document}